\documentclass{article}

\usepackage{amsmath}
\usepackage{url}
\usepackage{color}

\usepackage{amsfonts, amsthm, amsmath, amssymb}

\theoremstyle{plain}
\newtheorem{thm}{Theorem}[section]
\newtheorem{lem}[thm]{Lemma}
\newtheorem{prop}[thm]{Proposition}

\newtheorem{ques}[thm]{Question}

\theoremstyle{definition}
\newtheorem{defn}[thm]{Definition}

\newtheorem{ex}[thm]{Example}

\theoremstyle{remark}
\newtheorem{rem}[thm]{Remark}

\def\ds{\displaystyle}

\DeclareMathOperator{\N}{\mathbb N}

\DeclareMathOperator{\dd}{d}
\DeclareMathOperator{\dx}{d_{\max}}

\newcommand{\pre}{\preceq}
\newcommand{\st}{\ensuremath{\, | \,}}

\newcommand{\Sgen}{S=\langle a_1,...,a_t \rangle}
\def\red#1 {\textcolor{red}{#1 }}
\def\blue#1 {\textcolor{blue}{#1 }}
\begin{document}

\title{Maximal Denumerant of a Numerical Semigroup with Embedding Dimension Less Than Four}
\author{Lance Bryant, James Hamblin, and Lenny Jones}
\date{}
\maketitle

\begin{abstract}
\noindent Given a numerical semigroup $S = \langle a_1, a_2, \ldots, a_t \rangle$ and $s\in S$, we consider the factorization $s = c_1 a_1 + c_2 a_2 + \cdots + c_t a_t$ where $c_i\ge0$.  Such a factorization is {\em maximal} if $c_1+c_2+\cdots+c_t$ is a maximum over all such factorizations of $s$. We show that the number of maximal factorizations, varying over the elements in $S$, is always bounded. Thus, we define $\dx(S)$ to be the maximum number of maximal factorizations of elements in $S$. We study maximal factorizations in depth when $S$ has embedding dimension less than four, and establish formulas for $\dx(S)$ in this case.
\end{abstract}

\section{Introduction}

Let $\N$ denote the nonnegative integers. A {\em numerical semigroup} $S$ is a subsemigroup of $\N$ that contains 0 and has a finite complement in $\N$. For two elements $u$ and $u'$ in $S$, $u \pre u'$ if there exists an $s\in S$ such that $u+s = u'$. This defines a partial ordering on $S$. The minimal elements in $S\setminus\{0\}$ with respect to this ordering form the unique {\em minimal set of generators of $S$}, which is denoted by $\{a_1,a_2,\dots, a_t\}$ where $a_1<a_2<\dots<a_t$.  {The numerical semigroup} $S=\{\sum_{i=1}^{t} c_ia_i \st c_i\ge 0\}$ is represented using the notation $\Sgen$. Since the minimal generators of $S$ are distinct modulo $a_1$, the set of minimal generators is finite. Furthermore, $S$ having finite complement in $\N$ is equivalent to  {  $\gcd\left(a_1,a_2,\ldots ,a_t\right)= 1$. } The cardinality, $t$, of the set of minimal generators of a semigroup $S$ is called the {\em embedding dimension of $S$}. The element $a_1$ is called the {\em multiplicity} of $S$. When $S\ne \N$, we have $2\le t\le a_1$. %When $\nu(S)= e_0(S)$, $S$ is said to have {\em maximal embedding dimension}.

By definition, if $s\in S$, then there exists a $t$-tuple of nonnegative integers $(c_1,c_2,\dots,c_t)$ such that $s = c_1 a_1 + c_2 a_2 + \cdots + c_t a_t$. We call $(c_1,c_2,\dots,c_t)$ a {\em factorization} of $s$. For two factorizations $(c_1,c_2,\dots,c_t)$ and $(d_1,d_2,\dots,d_t)$ of $s$, we say they are different if $c_i\ne d_i$ for some $1\le i\le t$. The {\em length} of a factorization $(c_1,c_2,\dots,c_t)$ is defined as $c_1 + c_2 + \cdots + c_t$. The set of factorizations of $s$, denoted by $\mathcal F(s)$, is precisely the set of nonnegative integer solutions of the equation $x_1 a_1 + x_2 a_2 + \cdots + x_t a_t=s$ and is therefore finite.

A basic arithmetic constant that measures the behavior of factorizations in a numerical semigroup is the cardinality of $\mathcal F(s)$, which is called the {\em denumerant} of $s$ in $S$. See \cite{R} for an exhaustive view of results related to the denumerant. Recently, there has been interest in the factorization theory of numerical semigroups and the insight it provides into the general theory of commutative monoids; for example, see \cite{ACHP, AG1, AG2, BCKR, CDHK, CGLM, CHK}. Here we consider a variation of the denumerant.

%and we will denote it by $\dd(s;S)$

\begin{defn}
The {\em maximal denumerant} of $s$ in $S$ is the number of factorizations of $s$ that have maximal length and is denoted by $\dd_{\max}(s; S)$. %Likewise, the {\em minimal denumerant} of $s$ in $S$ is the number of factorizations of $s$ that have minimal length and is denoted by $\dd_{\min}(s; S)$.
\end{defn}

Certainly, the maximal denumerant of $s$ in $S$ is less than or equal to its denumerant in $S$, and thus also finite. On the other hand, unlike the denumerant, as we vary over the elements in $S$, the maximal denumerant is always bounded. This is not difficult to see and will be proven in Theorem \ref{thm1}. Thus, we have the well-defined quantity given in the next definition.

\begin{defn}
The {\em maximal denumerant} of $S$ is $$\dd_{\max}(S)=\max_{s\in S}\{\dd_{\max}(s; S)\}.$$
\end{defn}

We will focus on computing $\dd_{\max}(S)$ when $S$ can be generated by three elements; in particular, when $S$ has embedding dimension three. When $S$ is (perhaps non-minimally) generated by $a_1$, $a_2$, and $a_3$, by letting $g=\gcd(a_2-a_1,a_3-a_1)$, $m=(a_2-a_1)/g$, and $n=(a_3-a_1)/g$, we can write
$$S=\langle a_1,a_1+ gm,a_1+ gn\rangle,$$
which leads to unexpectedly nice formulas. Theorems \ref{thm2} and \ref{thm3} will be proven in Section \ref{dx}.\\

\noindent {\bf Theorem 3.5.  }{\it
Let $0\le \alpha < mn$ such that $\alpha \equiv -a_1  \,\,\mathrm{mod} \,\,mn$. We have the following formula:
$$ \dx(S) = \begin{cases} \ds
\left\lceil\frac{a_1}{mn}\right\rceil, & \ds  \text{if } \alpha \in  \langle m,n\rangle \\ \\ \ds
\left\lceil\frac{a_1}{mn}\right\rceil +1, & \text{otherwise.}
\end{cases}$$ }

\noindent {\bf Theorem 3.6.  }{\it
If $x$ and $y$ are integers such that $mx+ny=a_1$, then we have the following formula:
$$\ds \dd_{\max}(S) = \left\lceil \frac{x}{n}\right\rceil + \left\lceil \frac{y}{m} \right\rceil.$$
 }

The motivation for such a variation of the denumerant is the consideration of length-preserving restrictions. For example, perhaps we are interested in factorizations of an element that have either maximal or minimal length. This might happen when working with the numerical semigroup ring $R = k[[t^{a_1}, t^{a_2},\dots,t^{a_d}]]$, where $k$ is a field and $\mathbf m = (t^{a_1}, t^{a_2},\dots,t^{a_d})R$ is the unique maximal ideal. In this case, the maximal length of the factorizations of $s\in S$ is the $\mathbf m$-adic order of $t^s \in R$, i.e., the largest power of $\mathbf m$ that contains $t^s$, see \cite{B}. Another instance occurs in money-changing problems where the minimal length of the factorizations of $s\in S$ is the fewest number of coins needed to make change for $s$ using the denominations $a_1$, $a_2$, \dots, $a_t$, see \cite{BHJ}. Of course, the overarching concern is changing from one factorization to another in a numerical semigroup. The minimal presentation of a numerical semigroup (see \cite{RG}) is helpful when studying all factorizations; however, we note that it is not as useful for our current endeavor because these ``basic trades" do not generally preserve length. With the appropriate modifications, an approach via a minimal presentation may be fruitful and we leave this as an avenue for further research.

In the next section, we show that the maximal denumerant is always finite, and that for semigroups with embedding dimension less than three, $\dx(S) = 1$. In Section 3, we focus on numerical semigroups  with embedding dimension exactly equal to three. In the last section, we demonstrate the utility of our results by explicitly computing the maximal denumerant of semigroups with multiplicity 7 and embedding dimension 3.

\section{The Finiteness of $\dx(S)$}

For a given numerical semigroup $S=\langle a_1,a_2,\dots,a_t\rangle$, we need only find a finite set $U\subset S$ such that $\dd_{\max}(S)=\max_{s\in U}\{\dd_{\max}(s; S)\}$ to establish the finiteness of $\dx(S)$. To this end, we make the following definition:

%To show the finiteness of $\dx(S)$ for a given numerical semigroup $S=\langle a_1,a_2,\dots,a_t\rangle$, we need only find a finite set $U\subset S$ such that $\dd_{\max}(S)=\max_{s\in U}\{\dd_{\max}(s; S)\}.$

\begin{defn}
An element $u\in S$ is called {\em maximally reduced} if, for each $i$, with $1\le i\le t$, there exists a factorization $(c_1,c_2,\dots,c_t)$ of $u$ with maximal length such that $c_i=0$.
\end{defn}

We do not need $t$ distinct factorizations with maximal length to satisfy the definition of maximally reduced, as the next example shows.

\begin{ex}
In $S=\langle 7,8,13\rangle$, the element $48$ has the following factorizations:
\begin{itemize}
\item $(0,6,0)$
\item $(5,0,1)$
\item $(2,1,2)$
\end{itemize}
Notice that only the first two have maximal length. The first factorization with maximal length has a 0 in the first and third entries, and the second factorization with maximal length has a 0 in the second entry. Thus $48$ is a maximally reduced element.
\end{ex}

\begin{thm}\label{thm1}
Let $U$ be the set of maximally reduced elements in $S$. Then we have the following:
\begin{enumerate}
\item $U$ is a finite set
\item $\dd_{\max}(S)=\max_{s\in U}\{\dd_{\max}(s; S)\}.$
\end{enumerate}
Thus, $\dd_{\max}(S)$ is finite.
\end{thm}

\begin{proof}
To show that $U$ is a finite set, it suffices to prove that the maximally reduced elements are bounded above. To see this, set  {$N=(a_1-1)\sum_{i=2}^ta_i$}. Suppose that $s>N$, and  {that $C=(0,c_2,\dots,c_t)$ is a representation of $s$}. Then there exists  {$j$, with $2\le j\le t$,} such that $c_j \ge a_1$, and so $(a_j,c_2,\dots,c_j-a_1,\dots,c_t)$ is a representation of $s$ with greater length  {than $C$}. Therefore, every maximal representation has a first component that is nonzero and $s$ is not maximally reduced.

Now we need to show that $\dd_{\max}(S)=\max_{s\in U}\{\dd_{\max}(s; S)\}$. For $s\in S$, with maximal representations $\{C_j = (c_{j,1}, c_{j,2},\dots,c_{j,t})\}$, let $c_i=\min_j\{c_{j,i}\}$, and consider the element $s^*=s-\sum_{i=1}^t c_ia_i\in S$. Then it is not difficult to see that $s^*$ is maximally reduced and that $\dd_{\max}(s;S) =\dd_{\max}(s^*;S) $.
\end{proof}

Theorem \ref{thm1} outlines an algorithm for computing $\dd_{\max}(S)$: We check to see which elements up to $N=(a_1-1)\sum_{i=2}^ta_i$ are maximally reduced, and then take the maximum of the $\dd_{\max}(s; S)$ where $s$ is a maximally reduced element of $S$.

\begin{ex}
Let $S=\langle 7, 11, 13, 15\rangle$. Checking up to 234, the maximally reduced elements along with their maximal factorizations are
\begin{itemize}
\item $0$; $(0,0,0,0)$
\item $22$; $(0,2,0,0)$, $(1,0,0,1)$
\item $26$; $(0,0,2,0)$, $(0,1,0,1)$
\item $33$; $(0,3,0,0)$, $(1,0,2,0)$, $(1,1,0,1)$
\item $37$; $(0,1,2,0)$, $(0,2,0,1)$, $(1,0,0,2)$
\item $44$; $(1,2,0,1)$, $(1,1,2,0)$, $(0,4,0,0)$, $(2,0,0,2)$.
\end{itemize}
Therefore, $\dd_{\max}(S) = 4$.
\end{ex}

We see from the example that we can potentially improve this algorithm since we only need to check up to 44 to find the maximally reduced elements. We leave this improvement as an open question.

\begin{ques}
Can we improve the algorithm described in Theorem \ref{thm1}?
\end{ques}

In the next section we will focus on numerical semigroups with embedding dimension less than four, but first we consider the case when $S$ has embedding dimension strictly less than three. When $S=\mathbb N$, then every element $s\in S$ has a unique factorization, namely, $s = s \cdot 1$. Thus $\dx(\mathbb N) = 1$. We show that when $S$ has embedding dimension two, we also have that $\dx(S) = 1$.

\begin{prop}\label{for 2}
If $S=\langle a_1, a_2\rangle$, then $\dx(S)=1$.
\end{prop}

\begin{proof}
We will show that every element of $S$ has only one maximal factorization. Suppose that
\begin{align*}s=c_1a_1+c_2a_2 &\quad\text{ and }\quad s=d_1a_1+d_2a_2,\end{align*}
where $c_1+c_2=d_1+d_2$. If $c_1=d_1$ or $c_2=d_2$, then it follows that both $c_1=d_1$ and $c_2=d_2$. If this is not the case, then we may assume without loss of generality that $c_1>d_1$. But then we have $(c_1-d_1)a_1+c_2a_2=d_2a_2$ and
\begin{align*}
(c_1-d_1)a_1+c_2a_2 &< (c_1-d_1)a_2+c_2a_2\\
&= (c_1-d_1+c_2)a_2\\
&= d_2a_2.
\end{align*}
 This is a contradiction. Since we cannot have two factorizations of an element of $S$ with the same length, we certainly cannot have two with maximal length.
 \end{proof}

\section{The maximal denumerant of a semigroup with embedding dimension less than four}\label{dx}

%The key results in this section are Theorems \ref{prop ms char} and \ref{thm main1}. Theorem \ref{prop ms char} refines the relationship between $\dx(S)$ and the maximally reduced elements of $S$ when $S$ is 3-generated. This enables us to establish the formula for computing $\dx(S)$ given in Theorem \ref{thm main1}.

Throughout this section, unless otherwise stated, we assume that $S=\langle a_1,a_2,a_3\rangle$ is a numerical semigroup with embedding dimension three.  Set $g=\gcd(a_2-a_1,a_3-a_1)$,  $m=(a_2-a_1)/g$ and $n=(a_3-a_1)/g$. Then
\begin{eqnarray}\label{set} S=\langle a_1,a_1+mg,a_1+ng\rangle,\end{eqnarray}
where  {$\gcd(m,n)=\gcd(a_1,g)=1$.}
In the following lemma, we determine the maximally reduced elements of $S$ and their maximal factorizations.

\begin{lem}\label{thm emb3 red}
 {Let $s$ be} a maximally reduced element of $S$. Then $s$ is a multiple of $n a_2$. Moreover, if $s=kn a_2$, then $\{pU+qV \st p,q\ge0$ and $p+q=k\}$ is the set of maximal factorizations of $s$ where $U=(0,n,0)$ and $V=(n-m,0,m)$ (using the standard addition and scalar multiplication of vectors).
 \end{lem}

\begin{proof}
The element $s=0$, which is always maximally reduced, has the unique (maximal) factorization $(0,0,0)$. Certainly, $s$ is a multiple of $n a_2$. It is also easy to verify that the rest of the theorem is satisfied in this case. Now we assume that $s>0$. Since $s$ is maximally reduced, there exists a maximal representation of $s$ with the first component equal to 0, say $D= (0,d_2,d_3)$. Suppose that $d_3\ne 0$.  {Since }there exists another maximal factorization {$C=(c_1,c_2,0)$} of $s,$
we have
\[s=d_2a_2+d_3a_3 > (d_2+d_3)a_2 = (c_1+c_2)a_2 \ge c_1a_1+c_2a_2 = s,\]
which is impossible. Thus, $d_3=0$ and $D=(0,d_2,0)$.

The element $s> 0$ also has another maximal factorization $E=(e_1,0,e_3)$ distinct from $D$. We now have
\begin{eqnarray}\label{eq1}
d_2a_2 &=& e_1a_1+e_3a_3,
\end{eqnarray}
and
\begin{eqnarray}\label{eq2}
d_2a_1 &=& e_1a_1+e_3a_1.
\end{eqnarray}
Subtracting Equation \ref{eq2} from Equation \ref{eq1} and dividing by $g$ yields $m d_2=n e_3$. Since $m$ and $n$ are relatively prime, we have that $d_2=kn$ for some $k>0$, and so $s=d_2a_2=kn a_2$. Therefore, $s$ is a multiple of $n a_2$.

Next we show that $\{pU+qV \st p,q\ge0$ and $p+q=k\}$ is the set of maximal representations of the maximally reduced element $s=kn a_2$. Notice that our proof of the first statement of the theorem shows that $kU=(0,kn,0)$ is a maximal factorization of $s$. It is not difficult to see that $pU+qV$, where $p,q\ge0$ and $p+q=k$, is also a factorization of $s$ having the same length as $kU$. Thus, all of these factorizations are maximal. We still need to show that no other maximal factorizations exist.

Let $C=(c_1,c_2,c_3)$ be a maximal factorization of $s$. Similar to before, using that $(0,kn,0)$ is a maximal factorization, we have
\begin{eqnarray}\label{eq4}
(kn-c_2)a_2 &=& c_1a_1+c_3a_3,
\end{eqnarray}
and
\begin{eqnarray}\label{eq5}
(kn-c_2)a_1 &=& c_1a_1+c_3a_1.
\end{eqnarray}
Subtracting Equation (\ref{eq4}) from Equation (\ref{eq5}) and dividing by $g$ yields $m (kn-c_2)=n c_3$. Since $m$ and $n$ are relatively prime, we have that $kn-c_2=k'n$ and $c_3 = k'm$ for some $0\le k'\le k$. It follows that $c_2 = (k-k')n$ and $c_1=k'(n-m)$. Therefore, we have that $C=(k-k')U + k'V$.
\end{proof}

From Lemma \ref{thm emb3 red}, we can precisely describe the set of maximally reduced elements. This is the content of the next theorem.

\begin{thm}\label{cor1}
There exists an integer $k\ge 0$ such that $\{0,na_2,2na_2,\dots,kna_2\}$ is the set of maximally reduced elements in $S$. Furthermore, $\dx(ina_2;S) = i+1$ for $0\le i\le k$.
\end{thm}

\begin{proof}
We already know that every maximally reduced element in $S$ is a multiple of $na_2$. Thus, for the first statement it suffices to assume that $s=hna_2$ is maximally reduced and show this implies that $s' = (h-1)na_2$ is also maximally reduced.

If $(0,(h-1)n,0)$ is not a maximal factorization of $s'$, then $s'$ has a factorization $C=(c_1,c_2,c_3)$ such that $c_1+c_2+c_3 > (h-1)n$. It follows that $(c_1,c_2+n,c_3)$ would be a factorization of $s$ with length greater than $hn$. Therefore, $(0,hn,0)$ is not a maximal factorization of $s$, and $s$ is not maximally reduced by Theorem~\ref{thm emb3 red}. From this contradiction, we conclude that indeed $(0,(h-1)n,0)$ is a maximal factorization of $s'$. Clearly, we also have that $((h-1)(n-m),0,(h-1)m)$ is a maximal factorization of $s'$ since it is a factorization with length $(h-1)n$. This shows that $s'$ is maximally reduced.

Now, if $0\le i\le k$, then $ina_2$ is maximally reduced and by Theorem~\ref{thm emb3 red}, its maximal factorizations are $\{p(0,n,0)+q(n-m,0,m) \st p,q\ge0$ and $p+q=i\}$. Since $p$ can range from 0 to i (with $q$ depending on $p$), it follows that $\dx(ina_2;S) = i+1$.
\end{proof}

%Thereom \ref{cor1} reveals that $\dx(S)$ can be obtained from the set of maximally reduced elements in other ways, as the next immediate corollary indicates.

%Combining Theorem \ref{thm1} and Thereom \ref{cor1}, we see that if $U=\{0,na_2,2na_2,\dots,kna_2\}$ is the set of maximally reduced elements in $S$, then $\dx(S)=k+1$, i.e., the largest maximal denumerant of the maximally reduced elements. This also coincides with other measurements of the set $U$, which can now be easily verified.

%\begin{cor}\label{cor2}
%Let $U=\{0,na_2,2na_2,\dots,kna_2\}$ be the set of maximally reduced elements in $S$, then we have the following:
%\begin{enumerate}
%\item $\dx(S)$ is the cardinality of $U$
%\item $\dx(S)$ is the maximal denumerant in $S$ of the largest maximally reduced element.
%\end{enumerate}
%\end{cor}

%Both statements in Corollary \ref{cor2} fail to hold if $S$ is not 3-generated.
%\begin{ex} Example showing this.
%\end{ex}

Our main results, the formulas provided in Theorems \ref{thm2} and \ref{thm3}, are both dependent upon Lemma \ref{lemforthm2}. Notice that since $m$ and $n$ are relatively prime, $U=\langle m,n\rangle$ is a semigroup with embedding dimension less than three. It is well known that $U$ is a {\it symmetric} semigroup, i.e., for every $z\in \mathbb Z$, exactly one of $z$ or $f-z$ is in $U$, where $f$ is the Frobenius number of $U$. The Frobenius number of $U$ is the largest integer not in $U$, and we have $f=mn-m-n$. See \cite{FGH, RG} for more information concerning symmetric semigroups.

\begin{lem}\label{lemforthm2}
The following are equivalent:
\begin{enumerate}
\item[(a)] $hna_2$ is not a maximally reduced element of $S$
\item[(b)] $(0,hn,0)$ is not a maximal factorization of $hna_2$
\item[(c)] $hmn - a_1 \in \langle m,n\rangle.$
\end{enumerate}
Moreover, $\dx(S) = \min\{h \st hmn - a_1 \in \langle m,n\rangle\}$.
\end{lem}

\begin{proof}
For (a) implies (b), if $(0,hn,0)$ is a maximal factorization, then so is $(h(n-m),0,hm)$. Thus $hna_2$ is maximally reduced. For (b) implies (a), if $hna_2$ is maximally reduced, then by Lemma \ref{thm emb3 red}, $(0,hn,0)$ is a maximal factorization of $hna_2$.

For (b) implies (c), we may assume that $a_1 \in \langle m,n\rangle$, since otherwise we would have $a_1-m-n - (h-1)mn \not\in \langle m,n\rangle$. By the symmetry of $\langle m,n\rangle$, it follows that $hmn-a_1 \in \langle m,n\rangle$. By assumption we have that
\begin{equation}\label{jh1}
hna_2 = c_1a_1+c_2a_2+c_3a_3
\end{equation}
where $hn < c_1+c_2+c_3$.  Write $k = hn - (c_1+c_2+c_3)$.  Subtracting $(c_1+c_2+c_3)a_1$ from both sides of (\ref{jh1}), we have
\begin{equation*}
hngm-ka_1=c_2gm+c_3gn.
\end{equation*}
Since $\gcd(a_1,g)=1$, it follows that $k$ is divisible by $g$ and so
\[c_2m+c_3n=hmn-k'a_1,\] for some $k'> 0$.  Thus $hmn-k'a_1\in \langle m,n\rangle$, and since $a_1$ is as well, we have $hmn - a_1 \in \langle m,n\rangle$.

For (c) implies (b), we essentially reverse these steps. Since $hmn - a_1 \in \langle m,n\rangle$, we have $hmn-a_1 = c_2m+c_3n $ where $c_2,c_3\ge 0$. Multiplying both sides by $g$, adding $c_2a_1+c_3a_1$ to both sides and rearranging gives $$hna_2 = (hn+g-c_1-c_2)a_1 + c_2a_2+c_3a_3.$$ If we can verify that $hn+g-c_1-c_2\ge 0$, then $(0,hn,0)$ is not a maximal factorization of $hna_2$. To do this, suppose that $hn+g-c_1-c_2 < 0$. Then, using the fact that $m<n$, we get $mhn+mg < mc_2+mc_3 < mc_2 + nc_3 = hmn-a_1$. It follows that $a_2 = a_1+mg < 0$, which is a contradiction.
\end{proof}

Before we prove the main results, we consider when $S$ is generated by three elements, $a_1$, $a_2$, and $a_3$, that do {\em not} form the minimal generating set. In this case, $S$ has embedding dimension less than three, and we have seen that $\dx(S)=1$. The next lemma shows that we still have $\dx(S) = \min\{h \st hmn - a_1 \in \langle m,n\rangle\}$.

\begin{lem}\label{lemforthm22}
Let $S$ be a numerical semigroup generated by $a_1$, $a_2$, and $a_3$ such that these element do not form the minimal generating set of $S$. Then $\min\{h \st hmn - a_1 \in \langle m,n\rangle\} = 1$, or equivalently, $mn-a_1 \in \langle m,n\rangle$.
\end{lem}

\begin{proof}
Since $a_1$, $a_2$, and $a_3$ do not form the minimal generating set, we have either $a_2 = ka_1$ for $k\ge 2$ or $a_3 = pa_1 + qa_2$ where $p+q \ge 2$. In both case we can show that $a_1 < m+n$.

For the former case, we have $a_1 + gm = ka_1$. Thus $gm = (k-1)a_1$, and since $(a_1,g)=1$, it follows that $m = k'a_1$, where $1 \le k'=(k-1)/g \in \mathbb Z$. Therefore, $a_1 < m+n$. For the latter case, we have $a_1+ gn = pa_1 + q(a_1 + gm)$. Thus $g(n-qm) = (p+q-1)a_1$, and since $(a_1,g)=1$, it follows that $(n-qm) = k'a_1$, where $1\le k'=(p+q-1)/g \in \mathbb Z$. Therefore, $a_1 < m+n$.

Notice that if $a_1 < m+n$, then $a_1 - m-n \not\in \langle m,n\rangle$. By the symmetry of $\langle m,n\rangle$, we have $mn-m-n-(a_1-m-n) =mn-a_1 \in \langle m,n\rangle$.
\end{proof}

%\begin{thm}\label{thm2}
%If $\ds \left\lceil\frac{a_1}{mn}\right\rceil mn - a_1 \in \langle m,n\rangle$, then $\ds \dx(S) = \left\lceil\frac{a_1}{mn}\right\rceil$; otherwise, $\ds \dx(S) = \left\lceil\frac{a_1}{mn}\right\rceil+1.$
%\end{thm}

Next, we prove our main results with the following setting: The numerical semigroup $S$ is (perhaps non-minimally) generated by $a_1$, $a_2$, and $a_3$. Moreover, we set $g=\gcd(a_2-a_1,a_3-a_1)$,  $m=(a_2-a_1)/g$ and $n=(a_3-a_1)/g$ such that
\begin{eqnarray}\label{set} S=\langle a_1,a_1+mg,a_1+ng\rangle,\end{eqnarray}
where  {$\gcd(m,n)=\gcd(a_1,g)=1$.}

\begin{thm}\label{thm2}
Let $0\le \alpha < mn$ such that $\alpha \equiv -a_1  \,\,\mathrm{mod} \,\,mn$. We have the following formula:
$$ \dx(S) = \begin{cases} \ds
\left\lceil\frac{a_1}{mn}\right\rceil, & \ds  \text{if } \alpha \in  \langle m,n\rangle \\ \\ \ds
\left\lceil\frac{a_1}{mn}\right\rceil +1, & \text{otherwise.}
\end{cases}$$
\end{thm}

\begin{proof}
First note that $\ds \alpha = \left\lceil\frac{a_1}{mn}\right\rceil mn - a_1$ and $0 \le \ds \left\lceil\frac{a_1}{mn}\right\rceil mn - a_1 < mn$. Thus if $h<  \ds \left\lceil\frac{a_1}{mn}\right\rceil$, then $ hmn - a_1 \notin \langle m,n\rangle$. On the other hand, if $h>  \ds \left\lceil\frac{a_1}{mn}\right\rceil$, then $ hmn - a_1 > mn$ and thus $hmn-a_1$ is an element of $\langle m,n\rangle$.

The result now follows from Lemmas \ref{lemforthm2} and \ref{lemforthm22}. If $\ds\left\lceil\frac{a_1}{mn}\right\rceil mn - a_1 \in \langle m,n\rangle$, then $\ds\dx(S) = \min\{h \st hmn - a_1 \in \langle m,n\rangle\} = \left\lceil\frac{a_1}{mn}\right\rceil$. Otherwise, $\ds\dx(S) = \min\{h \st hmn - a_1 \in \langle m,n\rangle\} = \left\lceil\frac{a_1}{mn}\right\rceil+1$.

%If $\ds \left\lceil\frac{a_1}{mn}\right\rceil mn - a_1 \in \langle m,n\rangle$, then $\ds \left(\left\lceil\frac{a_1}{mn}\right\rceil -1\right)n$ is the largest maximally reduced element of $S$, and hence $\ds \dx(S) = \left\lceil\frac{a_1}{mn}\right\rceil$. Otherwise $\ds \left\lceil\frac{a_1}{mn}\right\rceil n$ is the largest maximally reduced element of $S$, and hence $\ds \dx(S) = \left\lceil\frac{a_1}{mn}\right\rceil +1$.
\end{proof}

\begin{thm}\label{thm3}
If $x$ and $y$ are integers such that $mx+ny = a_1$, then
$$\dx(S) = \left\lceil\frac{x}{n}\right\rceil + \left\lceil\frac{y}{m}\right\rceil.$$
\end{thm}

\begin{proof}
Note that $a_1 = mx+ny = mu+nv$ implies that $m(x-u)=n(v-y)$.  Since $\gcd(m,n)=1$, we have that $x-u=kn$ and $v-y=k m$ for some integer $k$.

%First, $a_1=mu+nv$ for some integers $u$ and $v$ if and only if $u=x+kn$ and $v=y-km$ for some integer $k$.
Thus we see that
\begin{align*} \left\lceil\frac{u}{n}\right\rceil + \left\lceil\frac{v}{m}\right\rceil &= \ds\left\lceil\frac{x+kn}{n}\right\rceil + \left\lceil\frac{y-k m}{m}\right\rceil\\
&= \ds\left\lceil\frac{x}{n}+k\right\rceil + \left\lceil\frac{y}{m}-k\right\rceil\\
&= \ds\left\lceil\frac{x}{n}\right\rceil+k + \left\lceil\frac{y}{m}\right\rceil-k\\
&= \ds\left\lceil\frac{x}{n}\right\rceil+ \left\lceil\frac{y}{m}\right\rceil.
\end{align*}
In other words, the formula is independent of the linear combination that we choose.

Now let $k = \dx(S)$, so that,  by Lemmas \ref{lemforthm2} and \ref{lemforthm22}, we have $k=\min\{h \st hmn - a_1 \in \langle m,n\rangle\}$. Thus, $kmn-a_1 = c_1m + c_2n$ for some $c_1,c_2 \ge 0$. Furthermore, $c_1 < n$ and $c_2 < m$ since otherwise, $(k-1)mn - a_1 \in \langle m,n\rangle$. We now have that
\begin{eqnarray*}
\left\lceil\frac{x}{n}\right\rceil + \left\lceil\frac{y}{m}\right\rceil &=& \left\lceil\frac{kn-c_1}{n}\right\rceil + \left\lceil\frac{-c_2}{m}\right\rceil\\
&=& k + \left\lceil\frac{-c_1}{n}\right\rceil + \left\lceil\frac{-c_2}{m}\right\rceil\\
&=& k\\
&=& \dx(S).
\end{eqnarray*}\end{proof}

Of course, the formulas presented in Theorems \ref{thm2} and \ref{thm3} are most interesting when $S$ has embedding dimension three, since otherwise, we know that $\dx(S) =1$. However, the fact that these formulas work for all numerical semigroups with embedding dimension less than four naturally raises the following question:

\begin{ques}
When $S$ has embedding dimension less than $t+1$, where $t\ge 4$, do there exist formulas dependent upon a set of $t$ generators analogous to those in Theorems \ref{thm2} and \ref{thm3} that yield the maximal denumerant of $S$?
\end{ques}

\

\section{The maximal denumerant of basic semigroups}

We begin with the following proposition that will aid in some computations.

\begin{prop}\label{propex1}
Let $S=\langle a_1,a_2,a_3\rangle$ be a semigroup. If either $m$ or $n$ divides $a_1$, then $\dx(S) = \left\lceil\frac{a_1}{mn}\right\rceil$.
\end{prop}

\begin{proof}
Assume that $m$ divides $a_1$. Then $a_1 = km $ for some $k>0$, and by Theorem \ref{thm3} we have $\dx(S) = \left\lceil\frac{k}{n}\right\rceil = \left\lceil\frac{km}{nm}\right\rceil =\left\lceil\frac{a_1}{mn}\right\rceil$. We have a similar proof whenever $n$ divides $a_1$.
\end{proof}

Recall our setting: $S=\langle a_1,a_1+mg,a_1+ng\rangle$, where $g=\gcd(a_2-a_1,a_3-a_1)$,  $m=(a_2-a_1)/g$ and $n=(a_3-a_1)/g$. From Theorems \ref{thm2} and \ref{thm3}, we see that if $T=\langle a_1,a_1+m,a_1+n\rangle$, then $\dx(S)=\dx(T)$. Thus, we will restrict our attention to the following class of numerical semigroups.

\begin{defn}
The semigroup $S=\langle a_1,a_2,a_3\rangle$ is called {\em basic} if $\gcd(a_2-a_1,a_3-a_1) = 1$.
\end{defn}

%The following proposition is not exhaustive, but for a given value of $a_1$ covers all but finitely many cases.

\begin{prop}\label{propex2}
Let $S=\langle a_1,a_2,a_3\rangle$ be a basic semigroup. Then
\begin{enumerate}
\item If $4a_1 = 2a_2+a_3$, then $\dx(S) = 2$
\item If $3a_1= a_2+a_3$, then $\dx(S) = 2$.
\item If $4a_1 <  2a_2+a_3$ and $3a_1 \ne a_2+a_3$, then $\dx(S) = 1$
\end{enumerate}
\end{prop}

\begin{proof}
For the first case, by subtracting $3a_1$, we have $a_1 = 2m+n$. Thus, $\dx(S) = \left\lceil\frac{2}{n}\right\rceil + \left\lceil\frac{1}{m}\right\rceil$. Since $1\le m<n$, we have $\dx(S) = 2$. Similarly, for the second case, we obtain $a_1=m+n$. Thus $\dx(S) = \left\lceil\frac{1}{n}\right\rceil + \left\lceil\frac{1}{m}\right\rceil =2.$

For the third case, consider when $a_1<m+n$. Then $a_1-m-n <0$, and hence, is not in $\langle m,n\rangle$. By the symmetry of $S$, we have $mn-m-n -(a_1-m-n) = mn-a_1 \in \langle m,n\rangle$. By Lemma \ref{lemforthm2}, $\dx(S) = 1$. On the other hand, if we have $m+n < a_1 < 2m+n$, then $0< a_1-m-n < m$. Again we have that $a_1-m-n \not\in \langle m,n\rangle$, and it follows that $\dx(S) = 1$.
\end{proof}
\begin{rem}
  Note that Proposition \ref{propex2} is not exhaustive in the sense that when $4a_1>2a_2+a_3$, there are more possibilities to consider. Nevertheless, for a given $a_1$, Proposition \ref{propex2} does address all but finitely many situations.
\end{rem}
%\begin{proof}
%since we have $a_1 < m+n$, in the equation $mx+ny = a_1$ we can assume that $0\le y<m$ and $x < 0$. Thus we have $\left\lceil\frac{x}{n}\right\rceil + \left\lceil\frac{y}{m}\right\rceil = 1$.
%\end{proof}

The next example demonstrates how we can easily compute the maximal denumerant of all basic semigroups with a fixed multiplicity.

\begin{ex}
Let $S=\langle 7,a_2,a_3\rangle$ be a basic semigroup with multiplicity $a_1=7$. Using Proposition \ref{propex2}, we carry out the following steps:
\begin{enumerate}
\item Solve $28 = 2a_2 + a_3$ to get the pairs of solutions $(8,12)$ and $(9,10)$.
\item Solve $21 = a_2+a_3$ to get the pairs of solutions $(8,13)$, $(9,12)$, and $(10,11)$.
\item Solve $4a_1 > 2a_2+a_3$ to get the pairs of solutions $(8,9)$, $(8,10)$, and $(8,11)$.
\end{enumerate}

The semigroups from the first two steps will have maximal denumerant equal to 2. The maximal denumerant of the semigroups in the third step can be computed using Proposition \ref{propex1}. All other basic semigroups with multiplicity 7 have maximal denumerant equal to 1. The list below summarizes our computations.

\begin{enumerate}
\item $\dx(\langle 7,8,9\rangle) = 4$
\item $\dx(\langle 7,8,10\rangle) = 3$
\item $\dx(S) = 2$ if $S$ is one of the following:
\begin{enumerate}
\item $\langle 7,8,11\rangle$
\item $\langle 7,8,12\rangle$
\item $\langle 7,9,10\rangle$
\item $\langle 7,8,13\rangle$
\item $\langle 7,9,12\rangle$
\item $\langle 7,10,11\rangle$
\end{enumerate}
\item $\dx(S) = 1$ otherwise.
\end{enumerate}
\end{ex}

This example raises a natural question: for which values of $m$ and $n$ is $\dx(\langle a_1, a_1+m, a_1+n \rangle)$ maximized?  Considering Theorem 3.5, if $m=1$, then $\langle m, n\rangle = \N$, and $\dx(S) = \lceil a_1/n \rceil$.  This is maximized when $n=2$.  If $m>1$, then since $a_1 \geq 3$ we have $\lceil a_1/mn \rceil +1 \le \lceil a_1/2 \rceil$.  Hence $\dx(S)$ is largest when $m=1$ and $n=2$.


\begin{thebibliography}{10}

\bibitem{ACHP} Amos, J.; Chapman, S. T.; Hine, N.; Paix\~ao, J. Sets of lengths do not characterize numerical monoids. \emph{Integers} {\bf 7} (2007), A50.

\bibitem{AG1} Aguil\'{o}-Gost, F.; Garc\'{i}a-S\'{a}nchez, P. A. Factorization and catenary degree in 3-generated numerical semigroups. European Conference on Combinatorics, Graph Theory and Applications (EuroComb 2009), 157-161, \emph{Electron. Notes Discrete Math.}, {\bf 34}, Elsevier Sci. B. V., Amsterdam.

\bibitem{AG2} \textemdash, Factoring in embedding dimension three numerical semigroups. \emph{Electron. J. Combin.} {\bf 17} (2010).

\bibitem{BCKR} Bowles, C.; Chapman, S. T.; Kaplan, N.; Reiser, D. On delta sets of numerical monoids. \emph{J. Algebra Appl.} {\bf 5} (2006), {695-718}.

\bibitem{B}
Bryant, L. Goto numbers of a numerical semigroup ring and the Gorensteiness of associated graded rings, {\it J. Comm. Algebra.} {\bf 38:6}(2010), 2092-2128.

\bibitem{BHJ}
Bryant, L; Hamblin, J; Jones, L. A Variation on the Money-Changing Problem. {\it Am. Math. Mon.} {\bf 109} (2012), 406-414.


\bibitem{CDHK} Chapman, S. T.; Daigle, J.; Hoyer, R.; Kaplan, N. Delta sets of numerical monoids using nonminimal sets of generators. \emph{Comm. Algebra} {\bf 38} (2010), {2622-2634}.

\bibitem{CGLM} Chapman, S. T.; Garc\'{i}a-S\'{a}nchez, P. A.; Llena, D.; Marshall, J. Elements in a numerical semigroup with factorizations of the same length. \emph{Canad. Math. Bull.} {\bf 54} (2011), {39-43}.

\bibitem{CGLPR} Chapman, S. T.; Garc\'{i}a-S\'{a}nchez, P. A.; Llena, D.; Ponomarenko, V.; Rosales, J. C. The catenary and tame degree in finitely generated commutative cancellative monoids. \emph{Manuscripta Math.} {\bf 120} (2006), {253-264}.

\bibitem{CHK} Chapman, S. T.; Hoyer, R.; Kaplan, N. Delta sets of numerical monoids are eventually periodic. \emph{Aequationes Math.} {\bf 77} (2009), {273-279}.

\bibitem{FGH} Fr\"oberg, R.; Gottlieb, C.; H\"aggkvist R. On Numerical Semigroups. {\it Semigroup Forum} {\bf 35} (1987), 63-83.

\bibitem{RG} Garc\'{i}a-S\'{a}nchez, P. A.; Rosales, J. C. Numerical semigroups. Developments in Mathematics, {\bf 20}. \emph{Springer, New York}, 2009.

\bibitem{R}
 Ram\'irez Alfons\'in, J. L. \emph{The Diophantine Frobenius Problem}, Oxford Lecture Series in Mathematics and its Applications {\bf 30}, Oxford University Press, Oxford, (2005).

\end{thebibliography}
\end{document}